\documentclass[10pt]{amsart}

\usepackage{amsmath,graphics}
\usepackage{amsfonts,amssymb}
\usepackage{amsmath,amscd}
\usepackage[OT2,T1]{fontenc}
\DeclareMathAlphabet{\mathcalligra}{T1}{calligra}{m}{n}
\DeclareSymbolFont{cyrletters}{OT2}{wncyr}{m}{n}
\DeclareMathSymbol{\Sha}{\mathalpha}{cyrletters}{"58}

\theoremstyle{plain}
\newtheorem*{theorem*}{Theorem}
\newtheorem*{lemma*} {Lemma}
\newtheorem*{corollary*} {Corollary}
\newtheorem*{proposition*} {Proposition}
\newtheorem{theorem}{Theorem}[section]
\newtheorem{lemma}[theorem]{Lemma}

\theoremstyle{remark}

\newtheorem*{definition}{Definition}

\theoremstyle{definition}

\begin{document}

\author{Aftab Pande}
\address{Instituto de Matem\'atica, Universidade Federal do Rio de Janeiro, Brasil}
\email{aftab.pande@gmail.com}
\title{Large images of reducible galois representations}

\maketitle

\begin{abstract} Given a reducible Galois representation $\overline{\rho}: G_{\mathbb{Q}} \rightarrow GL_2( \mathbb{F}_q)$ we show there exists an irreducible deformation $\rho : G_{\mathbb{Q}} \rightarrow GL_2 (\mathbb{W} [[T_1, T_2,.., T_r,....,]])$  of $\overline{\rho}$ ramified at infinitely many primes, where $\mathbb{W}$ denotes the ring of Witt vectors of  $\mathbb{F}_q$. This is a modification of Ramakrishna's result for the irreducible case.
\end{abstract}

\footnote{AMS MSC codes 11F80. Keywords: Deformations; Galois representations.}

\section{Introduction}

In \cite{R} it was shown that one could lift a mod $p$ representation $\overline{\rho}$ to a power series ring in infinitely many variables which was generalized for totally real fields by \cite{P}. In this paper, we extend these results for a reducible representation  $\overline{\rho}: G_{\mathbb{Q}} \rightarrow GL_2(\mathbb{F}_q)$, where $\mathbb{F}_q$ is a finite field of residue characteristic $p$ and cardinality $q = p^t$.  We use cohomology classes which work for all lifts $\rho_n$ unlike \cite{HR} where their cohomology classes cannot be used to lift from mod $p$ to mod $ p^2$. This allows us to get an irreducible deformation of a reducible representation in infinitely many variables. The case of a reducible deformation of a residually reducible representation was addressed in \cite{S}. The author hopes to use these methods to generalize other lifting results of Ramakrishna for arbitrary number fields in an ongoing project.

Our main theorem is the following:

\begin{theorem} \label{t1}
Let $\overline{\rho}: G_{\mathbb{Q}} \rightarrow GL_2( \mathbb{F}_q)$ where $\overline{\rho} = \left(
                  \begin{array}{cc}
                    \phi & * \\
                    0 & 1 \\
                  \end{array}
                \right)$ and $S$ be the set of primes containing $p$ and $\infty$ and all those at which $\overline{\rho}$ is ramified. Suppose:
\begin{itemize}
\item $p \geq 3$
\item $\overline{\rho}$ is indecomposable
\item the $\mathbb{F}_p$ span of the elements in the image of $\phi$ is all of $\mathbb{F}_q$,
\item $\phi^2 \ne1$
\item $\phi \ne \chi, \chi^{-1}$, where $\chi$ is the mod $p$ reduction of the cyclotomic character
\item for $\overline{\rho}$ odd that $\overline{\rho} |_{G_p}$ is not unramified of the form $ \left(
                  \begin{array}{cc}
                    1 & * \\
                    0 & 1 \\
                  \end{array}
                \right)$, and for $\overline{\rho}$ even that $\overline{\rho} |_{G_p}$ is not $ \left(
                  \begin{array}{cc}
                    \chi & 0 \\
                    0 & 1 \\
                  \end{array}
                \right)$ or $ \left(
                  \begin{array}{cc}
                    \chi^{-1} & * \\
                    0 & 1 \\
                  \end{array}
                \right)$, where the $*$ may be trivial.
                
\end{itemize}
then there exists an irreducible deformation $\rho : G_{\mathbb{Q}} \rightarrow GL_2 (\mathbb{W} [[T_1, T_2,.., T_r,....,]])$  of $\overline{\rho}$ ramified at infinitely many primes, where $\mathbb{W}$ denotes the ring of Witt vectors of  $\mathbb{F}_q$.
\end{theorem}

 We start with $\overline{\rho} : G_{\mathbb{Q}} \rightarrow GL_2 (\mathbb{F}_q)$ and by adding primes to the ramification we lift it successively to $\rho_n : G_{\mathbb{Q}, S_n} \rightarrow GL_2 (\mathbb{W}[[ T_1,...,T_n]] / (p, T_1,...,T_n)^n)$, and define $\rho = \displaystyle\lim_{\overleftarrow{n}} \rho_n$. If $R_n$ is the deformation ring of $\rho_n$ with $m_{R_n}$ its maximal ideal, then we see that $R_n / m^n_{R_n} = \mathbb{W}[[ T_1,...,T_n]] / (p, T_1,...,T_n)^n$. We will add more primes of ramification to $S_n$ and get a new set of primes $S_{n+1}$, such that the deformation ring associated to $S_{n+1}$ has $R_{n+1}/m^{n+1}_{R_{n+1}}$ as a quotient. This gives us a surjection from $R_{n+1} / m_{R_{n+1}}^{n+1} \twoheadrightarrow R_n /m_{R_n}^n$, which allows us to get the inverse limit $R = \displaystyle\lim_{\overleftarrow{n}} R_n/ m_{R_n}^n$.

\section{Notation}

We refer the reader to the notation used in \cite{HR} but briefly outline some definitions and notations here.

\begin{itemize}
\item $G_Z$ is the Galois group over $\mathbb{Q}$ of its maximal extension unramified outside a finite set of primes $Z$.
\item For $w \in Z$, $G_w = Gal (\overline{\mathbb{Q}}_w/ \mathbb{Q}_w)$, where $\mathbb{Q}_w$ is the completion of $\mathbb{Q}$ at $w$.
\item For a $G_{\mathbb{Q}} = Gal (\overline{\mathbb{Q}}/\mathbb{Q})$ module $M$, $\mathbb{Q}(M)$ is the field fixed by the subgroup of  $G_{\mathbb{Q}}$ that acts trivially on $M$.
\item The $\mathbb{G}_m$-dual of $M$ is denoted by $M^*$.
\item For $f \in H^1(G_{\mathbb{Q}}, M)$, we denote by $L_f$ the field fixed by the kernel of the homomorphism $f |_{Gal ({\overline{\mathbb{Q}}/\mathbb{Q}(M)})}$.
\item $X = Ad^0 (\overline{\rho})$ is the set of trace zero $2 \times 2$ matrices over $\mathbb{F}_q$ with Galois action through $\overline{\rho}$ by conjugation.
\item Let $K = \mathbb{Q} (X^*)$ which is equivalent to $ \mathbb{Q}(X,\mu_p)$.
\item For $w$ unramified in a Galois extension $L /\mathbb{Q}$ we denote a Frobenius at $w$ by $\sigma_w$.
\item $S$ is the set of primes containing $p,\infty$ and all those at which $\overline{\rho}$ is ramified.
\item For a character $\kappa : G_{\mathbb{Q}} \rightarrow \mathbb{F}_q^*$, we denote by $\mathbb{F}_q (\kappa)$ the module $\mathbb{F}_q$ with Galois action via $\kappa$.
\end{itemize}

\section{Trivial primes and the modification of $N_v$}

We modify the lemmas in \cite{R} for a residually reducible representation $\overline{\rho}: G_{\mathbb{Q}} \rightarrow GL_2(\mathbb{F}_q)$ using the language of  \cite{HR} and some ideas from \cite{CP}. The following lemmas are used in the next section to find sets of  primes that we add to the ramification to remove global obstructions, and cohomology classes associated to these new primes which we use to overcome local obstructions to lifting at each level $n$. The lemmas are adaptions of lemmas of Ramakrishna, so we show the modification and outline the rest of the argument.

\begin{definition}
Let $\overline{\rho}$ be as in the hypothesis of Theorem \ref{t1}. For $v$ unramified in $\overline{\rho}$ we say $v$ is a trivial prime if:

\begin{itemize}
\item $v$ is unramified in $\mathbb{Q}(\overline{\rho})$ and $\overline{\rho}(\sigma_v)$ is trivial, and
\item $v \equiv 1 \mod p$
\end{itemize}
\end{definition}

Since $\overline{\rho}$ is reducible, the Galois module $X=Ad^0(\overline{\rho})$ has a filtration of Galois stable $\mathbb{F}_q$-subspaces of the form $U_1 = \left(\begin{array}{cc}
                     0 & b \\
                    0 & 0\\
                  \end{array}
                \right), U_2  =\left(\begin{array}{cc}
                    a & b \\
                    0 & -a \\
                  \end{array}
                \right), U_3 =\left( \begin{array}{cc}
                    a & b \\
                    c & -a \\
                  \end{array}
                \right)$ , while the Galois module $X^*$ has a filtration of the $\mathbb{F}_q$-subspaces $V_1 = (X/U_2)^*, V_2 = (X/U_1)^*, V_3 = X^*$. For a subquotient $M$ of $X$ or $X^*$, the $\phi$, trivial, $\phi^{-1}$, $\chi \phi$, $\phi$, $\chi \phi^{-1}$ eigenspaces are the eigenspaces under the prime to $p$ action of $Gal(\mathbb{Q} (\phi, \mu_p)/\mathbb{Q})$ under a splitting of the long exact sequence 
                
                $$1 \rightarrow Gal(K/\mathbb{Q} (\phi, \mu_p)) \rightarrow Gal(K/\mathbb{Q}) \rightarrow Gal (\mathbb{Q} (\phi, \mu_p)/ \mathbb{Q}) \rightarrow 1$$

\begin{definition}
For any $M \in \{U_1, U_2, U_3, V_1, V_2, V_3\}$ and $Z$ a finite set of primes containing $S$ we define $\Sha_Z^i (M)$ to be the kernel of the map $H^i(G_Z, M) \rightarrow \oplus_{v \in Z} H^i(G_v, M)$.
\end{definition}                

\begin{definition}
Let $N_w$ be a subgroup of $ H^1(G_w, M)$ and let $N_w^*$ be its annihilator in $ H^1(G_w, M^*)$ under local Tate duality. Let $N = \{N_w \}_{w \in \mathbb{Z}_p}$. The Selmer group $H^1_N (G_Z, M)$ is the kernel of the restriction map:

$$ H^1(G_Z, M) \rightarrow \oplus_{w \in Z} H^1(G_w, M)/N_w$$ 

Let $N^* = \{ N^*_w \}_{w \in \mathbb{Z}_p}$. We define the dual Selmer group $H^1_{N^* }(G_Z, M^*)$ is the kernel of the other restriction map:

$$ H^1(G_Z, M^*) \rightarrow \oplus_{w \in Z} H^1(G_w, M^*)/N_w^*$$

\end{definition}

\begin{definition}
We say an element $f \in H^1(G_Z, X)$ (resp. $\psi \in H^1(G_Z, X^*)$) has rank $d$ if $d$ is the smallest number such that $f$ (resp. $\psi$) is in the image of the map $H^1(G_Z, U_d) \rightarrow H^1(G_Z, X)$ (resp. $H^1(G_Z, V_d) \rightarrow H^1(G_Z, X^*)$)
\end{definition}

\begin{lemma} \label{l1}
Let $M$ be any of the subspaces $\{U_1, U_2, U_3, V_1, V_2, V_3\}$, then there exists a finite set $Q_1$ of  trivial primes such that $\Sha^1_{S \cup Q_1}(M) = 0$.
\end{lemma}

\begin{proof}

We mimic the proof of Prop 13 of \cite{HR} and outline the argument.

Let $\psi \in H^1 (G_{\mathbb{Q}}, X)$ and $L_{\psi}$ be the field fixed by the kernel of $\psi$. Let $P$ be the subgroup of $G_{\mathbb{Q}}$ that acts trivially on $M$ and $H= Gal(L_{\psi}/ \mathbb{Q})$. By Proposition 8 of \cite{HR}, we can assume that $H^1(Gal(\mathbb{Q}(M)/\mathbb{Q}),M)$ is trivial. We have the inflation-restriction sequence:
 
 $$0 \rightarrow H^1(H/P, X^P) \rightarrow H^1(H, M) \rightarrow H^1 (P, M)^{H/P}$$
 
Now, $H/P = Gal(\mathbb{Q}(M)/\mathbb{Q})$ and $P$ acts trivially on $X$, so $H^1(H/P, M^P) = H^1(Gal(\mathbb{Q}(M)/\mathbb{Q}),M)$ which was assumed to be trivial. So a non-trivial $\psi \in H^1(H,M)$ gives rise to a non-trivial element of $H^1 (P, M)^{H/P}$ which is $ Hom(P,M)^{H/P}$ as $P$ acts trivially on $M$. This shows that $L_{\psi}$ is a non-trivial extension of $\mathbb{Q}(M)$. If $\psi \in \Sha^1_S (M)$, then we choose a trivial prime $q$ such that it splits completely from $\mathbb{Q}$ to $\mathbb{Q}(M)$ but not from $\mathbb{Q}(M)$ to $L_{\psi}$, which means that $\psi |_{G_{q}} \ne 0$, so $\psi \notin \Sha^1_{S \cup \{q\}}(M)$. As $H^1(G_S, M)$ is finite, we repeat this procedure and get a finite set of trivial primes $Q_1$ such that $\Sha^1_{S \cup Q_1}(M) =0$.

\end{proof}

Let $(z_v)_{v \in S \cup Q_1} \notin \oplus_{v \in S \cup Q_1} N_v \oplus \psi_{S \cup Q_1} (H^1(G_{S \cup Q_1}, X))$ be a set of cohomology classes which we will use eventually to overcome obstructions to lifting in the next section.

\begin{lemma} \label{l2}
Let $Q_1$ be as in lemma \ref{l1}. There exists a Cebotarev class $L$ of trivial primes such that 

\begin{itemize}
\item $\beta |_{G_v}=0$ for all $\beta \in H^1 (G_{S \cup Q_1}, U_i^*)$ for $i=1,2$ and for all $\beta \in H^1(G_{S \cup Q_1}, X)$
\item There exists an $\mathbb{F}_p$-basis $\{ \psi,\psi_1,..,\psi_r \}$ of $H^1(G_{S \cup Q_1},X^*)$ such that $\{\psi_1,..,\psi_r \}$ is a basis of $\psi_{S \cup Q_1}^{*-1}(Ann(z_{w})_{w \in S \cup Q_1})$, $\psi |_{G_v} \ne 0$ and $\psi_i |_{G_v} = 0$ for all $i \geq 1$.
\end{itemize}
Furthermore, there is for each $v \in L$, a rank $3$ element $h^v \in H^1 (G_{S \cup Q_1 \cup \{v\}}, X)$ and a decomposition group above $v$ such that $h^v |_{G_w} =(z_w)_{w \in {S \cup Q_1}}$ and $h^v(\tau_v) = \left(  \begin{array}{cc}
                    0 & 0 \\
                    s & 0 \\
                  \end{array}
                  \right)$ with $s \ne 0$.
\end{lemma}

\begin{proof}
The difference between the above lemma and Prop 34 of \cite{HR} is that we have added the additional condition of $\beta |_{G_v} = 0$ for all $\beta \in H^1(G_{S \cup Q_1}, X)$, where $X$ corresponds to $U_3$ in the notation above. This means that we need the prime $v$ to split completely in the $\phi, \phi^{-1}$ and identity eigenspaces which are disjoint from the $\chi /\phi$ eigenspace of $U_1^*$ and the $\chi / \phi$ and $\chi$ eigenspaces of $U_2^*$.  The modified definition of trivial primes imposes only splitting conditions and the only non-splitting condition in the hypothesis above is in the $\chi \phi$ eigenspace of $Gal(K_{\psi}/K)$, none of which are in $U_1^*, U_2^*$ and $X$. Now, following the argument of Prop 34 of \cite{HR} we see that $v$ comes from a Cebotarev condition.
\end{proof}

As $h^v(\tau_v) = \left(  \begin{array}{cc}
                    0 & 0 \\
                    s & 0 \\
                  \end{array}
                  \right)$ with $s \ne 0$, we define the sets $C_v$ and $N_v$ of \cite{CP}, to be the conjugates by the matrix $ \left(  \begin{array}{cc}
                    0 & 1 \\
                    1 & 0 \\
                  \end{array}
                  \right)$ for the primes $v$ that we add. 

We cannot control the behavior of $h^v$ at $\sigma_v$, so we add a pair of primes $v_1, v_2$ such that $h = -h^{v_1} +2 h^{v_2}$ has the appropriate image at Frobenius and $h |_{G_w} = z_w$ for $w \in S \cup Q_1$. Altering the definition of trivial primes still allows us to use the same techniques of \cite{HR} so we can use the following result (Theorem 41 of \cite{HR}).

\begin{theorem} \label{t3}
There is a set of two primes $\{v_1,v_2\}$ coming from the Cebotarev class $L$ in the previous lemma such that for $h = -h^{v_1} +2 h^{v_2}$ we can choose the values of $h(\sigma_{v_i})$ arbitrarily for $i=1,2$.
\end{theorem}

\section{Main theorem and its proof}

Let $C_l$ be the set of deformation classes of $\overline{\rho}$ to $\mathbb{W}$ satisfying 

$ \rho(\sigma_l) = \left(
                  \begin{array}{cc}
                    l & 0 \\
                    0 & 1 \\
                  \end{array}
                \right)$                and $ \rho(\tau_l) = \left(
                  \begin{array}{cc}
                    1 & * \\
                    0 & 1 \\
                  \end{array}
                \right)$

We define $u_1, u_2 \in H^1 (G_l,X)$ by:

$ u_1 (\sigma_l) = \left(
                  \begin{array}{cc}
                    0 & 1 \\
                    0 & 0 \\
                  \end{array}
                \right)$ and $ u_1 (\tau_l) = \left(
                  \begin{array}{cc}
                    0 & 0 \\
                    0 & 0 \\
                  \end{array}
                \right)$

$ u_2 (\sigma_l) = \left(
                  \begin{array}{cc}
                    0 & 0 \\
                    0 & 0 \\
                  \end{array}
                \right)$ and $ u_2 (\tau_l) = \left(
                  \begin{array}{cc}
                    0 & 1 \\
                    0 & 0 \\
                  \end{array}
                \right)$
                
                Note that these two cohomology classes are the same as in \cite{HR}. We refer the reader to the calculations of Lemma 4.1 in \cite{CP} to produce the third cohomology class $u_3$ to get a three dimensional subspace $N_l$ which preserves $C_l$.

Recall that $R_n$ is the deformation ring of $\rho_n$ with $m_{R_n}$ its maximal ideal. We assume that there exists $\rho_n : G_{S_n} \rightarrow GL_2 (R_n / m^n_{R_n})$, $\Sha^2_{S_n}(X) =0$ and $\dim H^1_N (G_{S_{n}}, X) = n$. By theorem \cite{t3} we can find a set of primes $B$ such that $\dim H^1_N (G_{S_n \cup B}, X) = n +1$ (we simply choose the $\alpha_i \in C_{v_i}$ in the proof of theorem \ref{t3}). Let $U$ be the deformation ring and $\rho_U$ be the deformation associated to the augmented set $S_n \cup B$, with the deformation conditions ($N_v, C_v$). If $B$ consists of primes such that $\rho_n |_{G_{v}} \in C_v$ for $v \in B$, we have a surjection $ \phi : U \twoheadrightarrow  R_n /m_{R_n}^n$ and we follow the argument as in \cite{R} or \cite{P}.

If $\rho_n |_{G_{v}} \notin C_v$ for $v \in B$, then we choose a set of cohomology classes $(z_{v})_{v \in S_n \cup B}$ such that the action of $z_{v}$ on $\rho_{n}|_{G_{v}}$ overcomes the local obstructions at $v \in S_n \cup B$. By theorem \ref{t3}  we can find a set $A$ of two primes and a cohomology class $h$ such that:

\begin{itemize}
 \item $\tilde{\rho_n} = (I+p^n h)\rho_n |_{G_q} \in C_q $  for $q \in A$ (no new obstructions at $A$)
 \item $h |_{G_{v}} = z_{v}$, for $v \in S_n \cup B$  ($h$ overcomes local obstructions at $S_n \cup B$)
\end{itemize}

We now show that adding this set of primes $A$ does not alter the dimension of the Selmer groups, hence does not add more variables to the ring of power series.

\begin{lemma} \label{l4}
For a set $A = \{v_1, v_2\}$ of two primes chosen as in theorem \ref{t3}  and $(z_v)_{v \in S_n \cup B} \notin \oplus_{v \in S_n \cup B} N_v \oplus \psi_{S_n \cup B} (H^1(G_{S_n \cup B}, X))$ we have $H^1_N(G_{S_n \cup B}, X) = H^1_N (G_{S_n \cup B \cup A}, X)$.
 \end{lemma}

\begin{proof}

We adapt the argument of Prop 4.1 of \cite{R}.

Recall that in Lemma \ref{l2} the trivial primes $v$ were chosen so that $\beta \in H^1 (G_T, X) \Rightarrow \beta |_{G_v} =0$. As $N_v$ is a three dimensional subspace including the zero cocyle $u_2$, we see that $\beta |_{G_v} =0 \Rightarrow \beta \in N_v$. Thus, $H^1_{N} (G_{S_n \cup B }, X) \subset H^1_{N} (G_{S_n \cup B \cup A}, X)$. 

Any element of $H^1_{N} (G_{S_n \cup B \cup A}, X) \setminus H^1_{N} (G_{S_n \cup B}, X)$ necessarily looks like $f + \alpha_1 h^{v_1} + \alpha_2 h^{v_2}$, where $f \in H^1 (G_{S_n \cup B}, X)$ and $h^{v_i}$ are as in theorem \ref{t3}. Since $\alpha_1 h^{v_1} + \alpha_2 h^{v_2} |_{G_v} = (\alpha_1 + \alpha_2) z_v \in f +N_v$ for $v \in S_n \cup B \cup A $, we see that $\alpha_1 + \alpha_2 = 0$. We know that $\alpha_1 (h^{v_1} - h^{v_2})|_{G_{q}}= 0$ for all $q \in S_n \cup B$ which means that $f|_{G_{q}} \in N_{q}$ for all $q \in S_n \cup B$ $\Rightarrow f \in H^1_N (G_{S_n \cup B}, X)$. We also know that $f|_{G_{v_{i}}} = 0$ for $i = 1,2$, so $f \in H^1_N (G_{S_n \cup B \cup A}, X)$. Thus, $\alpha_1 (h^{v_1} - h^{v_2}) \in H^1_N (G_{S_n \cup B \cup A}, X) \Rightarrow \alpha_1 ( h^{v_1} - h^{v_2} ) |_{G_{v_i}} \in N_{v_i}$, for $i = 1,2$.  We now look at the construction of the $h^{v_i}$ in the proof of the previous lemma to get a contradiction.

If $\rho_n |_{G_{v_i}} \notin C_{v_i}$, we choose $h |_{G_{v_i}} = -h^{v_1} + 2 h^{v_2} \notin N_{v_i}$ for $i=1,2$. This implies that $-A+2E \notin N_{v_1}$ while we have that $ h^{v_1} - h^{v_2}  |_{G_{v_i}} \in N_{v_1} \Rightarrow A-E \in N_{v_1}$. Combining these two conditions we get that $A, E \notin N_{v_1}$ so $\alpha_1 =0$ and $f \in H^1_N (G_{S_n \cup B \cup A}, X)$  which is a contradiction. A similar argument works for $N_{v_2}$.

If $\rho_n |_{G_{v_i}} \in C_{v_i}$ and $h^{v_i} (\sigma_{v_i}) \notin N_{v_i}$ i.e. $A \notin N_{v_1}$, then using the fact that $A-E \in N_{v_1}$ and $-A +2E \in N_{v_1}$, we get that $A \in N_{v_1}$, which is a contradiction.

(Note that if $A$ consists of only one prime, then the proof is exactly the same as in the first part of Prop 4.1 in \cite{R})

\end{proof}

Let $\tilde{W}$ be the deformation ring and $\rho_{\tilde{W}}$ associated to the augmented problem with deformation conditions ($N_v, C_v$). As $\tilde{\rho_n} |_{G_q} \in C_q$ for $q \in A$ we have a surjection $ \phi : \tilde{W} \twoheadrightarrow  R_n /m_{R_n}^n$, which means that for some $I_1$, we $\tilde{W} / I_1 = R_n / m_{R_n}^n$. As $\dim H^1_N (G_{ S_n \cup B}, X) = \dim H^1_N (G_{S_n \cup B \cup A}, X)  =n +1$, we see that as a ring $\tilde{W}$ consists of power series of $(n+1)$ variables. Thus, for some $I_2$, $\tilde{W}/I_2 = \mathbb{F}_q [[T_1,...,T_{n+1}]]/(T_1,...,T_{n+1})^2 $. Let $I = I_1 \cap I_2$, and define $W_0 = \tilde{W} /I$. 

Our goal is to get a deformation ring which has $R_{n+1}/m_{R_{n+1}}^{n+1}$ as a quotient. If $W_0$ is such a deformation ring, we are done. If not, we get a sequence:

$$R_{n+1}/m_{R_{n+1}}^{n+1} \twoheadrightarrow ... \twoheadrightarrow W_1 \twoheadrightarrow W_0$$

where the kernel at each stage has order $p$. We add more primes of ramification to $S_n \cup B \cup A$ so that the augmented deformation ring has $W_1$ as a quotient and keep iterating to get our required deformation ring.

As $W_0$ is a quotient of $\tilde{W}$, we let $\rho_{W_0}$ be the deformation induced by $\rho_{\tilde{W}}$. As $\rho_{W_0} |_{G_v} \in C_v$ for $v \in S_n \cup B \cup A$ we can lift $\rho_{W_0}$ to $W_1$. Let us call this deformation $\rho_{W_1}$. Iterating the same argument as for $\rho_n$ we can lift $\rho_{W_1}$ to $W_2$ by adding a suitable set of primes $A_1$ to the set of ramification allowing us to eventually find a deformation that has $R_{n+1} / m^{n+1}_{R_{n+1}} = \mathbb{W}[[ T_1,...,T_{n+1}]] / (p, T_1,...,T_{n+1})^{n+1}$ as a quotient. 
Now we are in a position to state the final theorem.

\begin{theorem}
There exists an irreducible deformation of $\overline{\rho}$, ramified at infinitely many primes, $\rho : G_{\mathbb{Q}} \rightarrow GL_2 (\mathbb{W} [[T_1, T_2,.., T_r,....,]])$.
\end{theorem}

\begin{proof}
We let  $\rho = \displaystyle\lim_{\overleftarrow{n}} \rho_n$ and see that at each stage $n$, $Im \rho_n \supseteq GL_2 (\mathbb{W}[[ T_1,...,T_n]]/(p, T_1,...,T_n)^n)$. Hence, we get our desired deformation. By Corollary 43 of \cite{HR}, the deformation is irreducible.
\end{proof}

\section{Concluding remarks}

\begin{itemize}
\item 
In \cite{P}, one could not generalize the results of \cite{R} for all number fields. One of the problems in using our definition of trivial primes is that when one adds them to the ramification set to solve the local condition property (finding an $h^v$ such that $h^v |_{G_w} =(z_w)_{w \in Z}$), the behavior at inertia is hard to control. In the reducible case one can use the subspaces $U_i$ to find a suitable $h^v$ but, in the irreducible case it is hard to guarantee the behavior of $h^v$ at inertia.
\item In \cite{R}, the image of the deformation is full, i.e,  $\rho$ contains $SL_2(\mathbb{Z}_p [[T_1, T_2,.., T_r,....,]])$ but requires that the image of the residual representation $\overline{\rho}$ contains $SL_2(\mathbb{Z} / p \mathbb{Z})$, which is not true in our case. Hence, we do not get that the image of our deformation is full.
\end{itemize}

\section{Acknowledgements}

The author would like to thank A. Pacetti and M. Camporini for many conversations on deformation theory during two visits to Buenos Aires in 2014 which were funded by MathAmSud project DGMFPGT. This article was finished during a one year stay in 2015 at Universite Paris-6 which was funded by the CNPQ grant PDE 200845/2014-4. The author gratefully acknowledges support from all the agencies and thanks the Universidad de Buenos Aires and Universite Paris-6 for their hospitality.

\end{document}